\documentclass{amsart}

\usepackage{amsfonts, amssymb, amsmath, eucal, verbatim, amsthm, amscd, enumerate}

\usepackage{graphicx}
\usepackage{framed}
\usepackage{tikz}
\usepackage{enumitem}

\newtheorem{theorem}{Theorem}[section]
\newtheorem{lemma}[theorem]{Lemma}
\newtheorem{proposition}[theorem]{Proposition}

\newtheorem{corollary}[theorem]{Corollary}
\theoremstyle{definition}

\theoremstyle{remark}

\newtheorem*{remarks}{Remarks}
\newtheorem*{notation}{Notation}

\newtheorem*{organisation}{Organisation}
\newtheorem*{acknowledgements}{Acknowledgements}

\newcommand{\vertiii}[1]{{\left\vert\kern-0.25ex\left\vert\kern-0.25ex\left\vert #1 
\right\vert\kern-0.25ex\right\vert\kern-0.25ex\right\vert}}

\numberwithin{equation}{section}

\usepackage{setspace}

\def\1{\textbf{\rm 1}}

\begin{document}

\date{\today}

\author[Bez]{Neal Bez}
\address[Neal Bez]{Department of Mathematics, Graduate School of Science and Engineering,
Saitama University, Saitama 338-8570, Japan}
\email{nealbez@mail.saitama-u.ac.jp}
\author[Lee]{Sanghyuk Lee}
\address[Sanghyuk Lee]{Department of Mathematical Sciences, Seoul National University, Seoul 151-747, Korea}
\email{shklee@snu.ac.kr}
\author[Nakamura]{Shohei Nakamura}
\address[Shohei Nakamura]{Department of Mathematics and Information Sciences, Tokyo Metropolitan University,
1-1 Minami-Ohsawa, Hachioji, Tokyo, 192-0397, Japan}
\email{nakamura-shouhei@ed.tmu.ac.jp}

\title[Maximal estimates with orthonormal initial data]{Maximal estimates for the Schr\"{o}dinger equation with orthonormal initial data}

\begin{abstract}
For the one-dimensional Schr\"odinger equation, we obtain sharp maximal-in-time and maximal-in-space estimates for systems of orthonormal initial data. The maximal-in-time estimates generalize a classical result of Kenig--Ponce--Vega and allow us obtain pointwise convergence results associated with systems of infinitely many fermions. The maximal-in-space estimates simultaneously address an endpoint problem raised by Frank--Sabin in their work on Strichartz estimates for orthonormal systems of data, and provide a path toward  proving our maximal-in-time estimates.
\end{abstract}

\maketitle

\section{Introduction and main results}
\subsection{Pointwise convergence for the Schr\"{o}dinger equation}
In one spatial dimension, consider the free Schr\"{o}dinger equation 
\begin{equation}\label{e:SchrodingerEq}
\left\{ \begin{array}{ll}
i\partial_t u + \partial^2_x u =0,\;\;\; (t,x) \in \mathbb{R}^{1+1}, \\
u(0,x) = f(x), 
\end{array} \right.
\end{equation}
whose solution we denote by $u(t,x) = e^{it\partial^2_x}f(x)$. The problem of identifying the smallest exponent $s>0$ for which 
\begin{equation}\label{e:carleson}
\lim_{t\to0} e^{it\partial^2_x} f(x) = f(x) \quad {\rm a.e.}\;\;\; x\in \mathbb{R}
\end{equation}
holds for all $f\in H^s(\mathbb{R})$ originated in the famous paper by Carleson \cite{carleson}. Here, $H^s(\mathbb{R})= (1-\partial_x^2)^{-\frac s2}L^2(\mathbb{R})$ is the inhomogeneous Sobolev space of order $s$. It follows from \cite{carleson} that \eqref{e:carleson} holds true as long as $s\ge\frac14$, and shortly afterwards the problem found a complete solution when Dahlberg and Kenig \cite{Dahlberg-Kenig} showed that $s\ge\frac14$ is necessary for \eqref{e:carleson}. 

The standard way to tackle  this pointwise convergence problem is to consider maximal-in-time estimates of the form 
\begin{equation*}
\big\|\sup_t |e^{it\partial_x^2} f| \big\|_{L^q_x} \le C\|f\|_{H^s}
\end{equation*}
or its local variants,  since standard arguments allow one to deduce \eqref{e:carleson} for all $f \in H^s(\mathbb{R})$. Whilst local space-time bounds of this type suffice for the purpose of deducing \eqref{e:carleson}, a particularly strong global form of such an estimate,
\begin{equation}\label{e:maxi}
\big\| \sup_{t \in \mathbb{R}} |e^{it\partial_x^2} f| \big\|_{L^4_x(\mathbb{R})} \le C\|f\|_{\dot{H}^{\frac{1}{4}}}
\end{equation}
was obtained by Kenig--Ponce--Vega \cite{KPV}. Here, $\dot{H}^s(\mathbb{R})= (-\partial_x^2)^{-\frac s2}L^2(\mathbb{R})$ is the homogeneous Sobolev space of order $s$. Since the estimate \eqref{e:maxi} is scaling invariant and  $\|f\|_{\dot{H}^{\frac{1}{4}}}\lesssim \|f\|_{H^\frac14}$,  we note that, by an elementary rescaling argument, \eqref{e:maxi} is equivalent to the corresponding estimate with the inhomogeneous norm $H^\frac14$ on the right-hand side.

Whilst the equation \eqref{e:SchrodingerEq} describes the behavior of a single quantum particle, the present work is motivated by recent investigations of Chen--Hong--Pavlovi\'{c} \cite{CHP-1,CHP-2} and Lewin--Sabin \cite{LewinSabinWP,LewinSabinScatt} into the dynamics of a system of infinitely many fermions. In the case of a finite number $N$ of particles (in one spatial dimension), such a system is modelled by $N$ orthonormal functions $u_1,\ldots,u_N$ in $L^2(\mathbb{R})$ satisfying a system of Hartree equations
\begin{equation} \label{e:Hartree_finite}
i\partial_tu_k = (-\partial_{x}^2 + w * \rho)u_k \qquad (k = 1,\ldots,N),
\end{equation}
where $\rho = \sum_{j=1}^N |u_j|^2$ represents the total density of particles and $w$ is an interaction potential. 

In this work, we initiate the study of the pointwise convergence problem for a system of infinitely many fermions. As we shall soon see, our progress in this direction hinges on establishing a generalization of \eqref{e:maxi} for orthonormal systems (possibly infinite) of initial data $(f_j)_j$. A natural form of such an estimate is
\begin{equation}\label{e:PureMax}
\bigg\|  \sum_j \nu_j |e^{it\partial_x^2} f_j|^2 \bigg\|_{L^{2}_xL^\infty_t(\mathbb{R}^{1+1})} \le C \| \nu \|_{\ell^\beta}
\end{equation}
with the constant $C$ independent of the orthonormal system $(f_j)_j$ in $\dot{H}^\frac14(\mathbb{R})$ and $\nu = (\nu_j)_j$ in $\ell^\beta$, where $\beta\ge1$. Clearly \eqref{e:PureMax} is equivalent to the square function estimate
\begin{equation*}
\bigg\| \bigg( \sum_j  |e^{it\partial_x^2} f_j|^2 \bigg)^{1/2} \bigg\|_{L^{4}_xL^\infty_t(\mathbb{R}^{1+1})} \le C \bigg( \sum_j \| f_j \|_{\dot{H}^\frac14}^{2\beta} \bigg)^{1/2\beta}
\end{equation*}
for orthogonal systems $(f_j)_j$ in $\dot{H}^\frac14(\mathbb{R})$, and this reduces to \eqref{e:maxi} in the case where the system of initial data consists of a single function. It is also apparent that \eqref{e:PureMax} follows from \eqref{e:maxi} with $\beta = 1$ via the triangle inequality, and that such estimates get stronger as we increase $\beta$. Our first main result establishes the optimal value of $\beta$, although we have to pay a small price to obtain such a sharp result in the sense that our estimates are of weak type.
\begin{theorem}\label{t:maximal}
The estimate
\begin{equation}\label{e:max}
\bigg\|  \sum_j \nu_j |e^{it\partial_x^2} f_j|^2 \bigg\|_{L^{2,\infty}_xL^\infty_t(\mathbb{R}^{1+1})} \le C \| \nu \|_{\ell^\beta}
\end{equation}
holds for all systems of orthonormal functions $(f_j)_j$ in $\dot{H}^\frac14(\mathbb{R})$ and $\nu = (\nu_j)_j$ in $\ell^\beta$ if and only if $\beta < 2$.
Moreover, the restricted weak-type estimate  
$$
\bigg\|  \sum_j \nu_j |e^{it\partial_x^2} f_j|^2  \bigg\|_{L^{2,\infty}_xL^\infty_t(\mathbb{R}^{1+1})} \le C \| \nu \|_{\ell^{2,1}}
$$
also fails. 
\end{theorem}
In the above statement, $L^{2,\infty}$ denotes weak $L^2$ and $\ell^{2,1}$ is a (sequence) Lorentz space; the reader may consult, for example, \cite{SteinWeiss} for further details.

As one would expect, a maximal estimate of the form \eqref{e:max} implies pointwise convergence of quantities of the form $\sum_j \nu_j |e^{it\partial_x^2} f_j|^2$ whenever the system $(f_j)_j$ is orthonormal in $\dot{H}^\frac14(\mathbb{R})$ and $(\nu_j)_j$ belongs to $\ell^\beta$ with $\beta < 2$. As we now turn to describe, a more natural formulation of such a result is in terms of the density function $\rho_{\gamma(t)}$ of the solution $\gamma(t)$ to the operator-valued Hartree-type equation 
\begin{equation}\label{e:Hartree-free}
\left\{
\begin{array}{ll}
i\partial_t\gamma=[-\partial^2_x,\gamma], \quad (t,x)\in \mathbb{R}^{1+1}, \\
\gamma|_{t=0}=\gamma_0.
\end{array} 
\right.
\end{equation}
Here, $\gamma_0$ is a self-adjoint and bounded operator on $L^2(\mathbb{R})$, $[\cdot,\cdot]$ denotes the commutator, and the solution $\gamma$ is given by $\gamma(t) = e^{-it\partial^2_x} \gamma_0 e^{it\partial^2_x}$. This is the free version of the operator-valued Hartree-type equation 
\begin{equation}\label{e:Hartree-general}
\left\{
\begin{array}{ll}
i\partial_t\gamma=[-\partial^2_x+w\ast\rho_\gamma,\gamma], \quad (t,x)\in \mathbb{R}^{1+1}, \\
\gamma|_{t=0}=\gamma_0
\end{array} 
\right.
\end{equation}
associated with \eqref{e:Hartree_finite}, $\rho_\gamma$ is the so-called the density function, formally defined by $\rho_\gamma(x) = \gamma(x,x)$ where (with the typical abuse of notation) $\gamma(x,y)$ is the integral kernel of $\gamma$. The operator-theoretic viewpoint allows one to rigorously formulate the problem in the case of infinitely many particles (see, for example, the discussion in \cite{LewinSabinWP}).

Here we consider the convergence of the solution $\gamma(t)$ to the initial data $\gamma_0$ in terms of pointwise convergence of the associated density functions 
\begin{equation}\label{e:carleson-density}
\lim_{t\to0} \rho_{\gamma(t)}(x) = \rho_{\gamma_0}(x)\quad {\rm a.e.}\;\;\; x\in\mathbb{R},
\end{equation} 
and we seek as large a class of initial data $\gamma_0$ as possible. The Schatten classes $\mathcal{C}^\beta(\mathcal{H})$, associated with a given Hilbert space $\mathcal{H}$, provide a natural setting in order to quantify progress on this problem thanks to their monotonicity property, $\mathcal{C}^{\beta_1}(\mathcal{H}) \subseteq \mathcal{C}^{\beta_2}(\mathcal{H})$ provided $\beta_1\le \beta_2$ (we refer the reader forward to Section \ref{section:prelims} for the definition of Schatten classes).

As a consequence of Theorem \ref{t:maximal}, we have the following. 
\begin{corollary}\label{c:pointwise}
If $\gamma_0 \in \mathcal{C}^{\beta}(\dot{H}^{\frac14}(\mathbb{R}))$ is self-adjoint with $\beta < 2$, then the density functions $\rho_{\gamma(t)}(x)$ and $\rho_{\gamma_0}(x)$ are well defined and satisfy \eqref{e:carleson-density}. 
\end{corollary}
Formal considerations indicate that if $\gamma_0 \in \mathcal{C}^{\beta}(\dot{H}^{\frac14}(\mathbb{R}))$, then
\[
\rho_{\gamma(t)} (x) = \sum_j \nu_j |e^{it\partial_x^2} f_j(x)|^2
\]
for an appropriate system $(f_j)_j$ of orthonormal functions in $\dot{H}^\frac14(\mathbb{R})$ and coefficients $(\nu_j)_j$ in $\ell^\beta$. In the infinite-rank case, some care is required to ensure that this is the case (and in what precise sense); we postpone such discussion to Section \ref{section:maximal}. The role of the maximal estimate in Theorem \ref{t:maximal} is to allow us to deduce pointwise convergence in the infinite-rank case from the finite-rank case. The statement that \eqref{e:carleson-density} holds for $\gamma_0 \in \mathcal{C}^1(\dot{H}^\frac14(\mathbb{R}))$ is equivalent to the following: 
\begin{equation}\label{e:abPW}
	\lim_{t\to0} |e^{it\partial_x^2} f(x)| = |f(x)| \quad {\rm a.e.}\;\;\; x\in \mathbb{R},  \quad  \forall  f\in \dot{H}^\frac14(\mathbb{R}).
\end{equation}
 Hence, recalling $\mathcal{C}^1(\dot{H}^{\frac14}(\mathbb{R})) \subset \mathcal{C}^\beta(\dot{H}^{\frac14}(\mathbb{R}))$ for $\beta>1$, the pointwise convergence \eqref{e:carleson-density} with some $\beta>1$ can be seen as a significant improvement of \eqref{e:abPW}.

\subsection{Strichartz estimates for orthonormal systems of data}
Next, we introduce our second main result in this paper concerning Strichartz estimates for the Schr\"odinger equation for orthonormal systems of initial data. The result is of interest for two reasons; firstly, it addresses a problem left open in recent work of Frank \emph{et al} \cite{FLLS} and Frank--Sabin \cite{frank-sabin-1,frank-sabin-2}, and secondly it provides a path to proving our maximal estimates in Theorem \ref{t:maximal}. We now seek to clarify these comments.

The classical Strichartz estimates for the one-dimensional free Schr\"{o}dinger propagator $e^{it\partial^2_x}$ state that for all $2\le q,r \le \infty $ satisfying $ \frac 2q + \frac 1r = \frac 12 $ (we say $(q,r)$ is an admissible pair in that case), the estimate
\begin{equation}\label{e:ClassicalStr}
\| e^{it\partial^2_x} f \|_{L^q_tL^r_x(\mathbb{R}^{1+1})} \le  C\|f\|_{L^2(\mathbb{R})}
\end{equation}
holds true for $f\in L^2(\mathbb{R})$; see, for example, \cite{GinibreVelo,KeelTao,Strichartz}. Recently, this classical setting was significantly generalized to estimates of the form
\begin{equation}\label{e:ONS}
\bigg\|  \sum_j \nu_j |e^{it\partial^2_x} f_j|^2 \bigg\|_{L^{q/2}_tL^{r/2}_x(\mathbb{R}^{1+1})} \le C \| \nu \|_{\ell^\beta}
\end{equation}
for systems of orthonormal functions $(f_j)_j$ in $L^2(\mathbb{R})$ and $\nu=(\nu_j)_j$ in $\ell^\beta$. For admissible pairs with \emph{finite} values of $r$, the optimal value of $\beta$ has been determined as follows.
\begin{theorem}[\cite{FLLS,frank-sabin-1}]\label{t:ONS-known} 
Let $q,r\ge2$ satisfy 
$$
\frac2q + \frac 1r = \frac 12,\;\;\; 2\le r < \infty.
$$ 
Then \eqref{e:ONS} holds if and only if $\beta \le \frac {2r}{r+2}$.
\end{theorem}
One can notice that the endpoint case $(q,r) = (4,\infty)$ is missing in Theorem \ref{t:ONS-known}. It was observed in \cite{FLLS} that $\beta < 2$ is necessary at the endpoint. On the other hand, as far as the authors are aware, there are no non-trivial positive results at the endpoint which improve upon the trivial case $\beta = 1$ (which follows from \eqref{e:ClassicalStr} and the triangle inequality). Moreover, it was conjectured by Frank and Sabin \cite{frank-sabin-2} that the desired endpoint estimate
\begin{equation}\label{e:ONS-1d}
\bigg\|  \sum_j \nu_j |e^{it\partial_x^2} f_j|^2   \bigg\|_{L^{2}_tL^\infty_x(\mathbb{R}^{1+1})} \le C \| \nu \|_{\ell^\beta}
\end{equation} 
should hold for some $\beta \in (1,2)$.
In a similar spirit to Theorem \ref{t:maximal}, we obtain the optimal range of $\beta$ for a weak-type version of \eqref{e:ONS-1d}.
\begin{theorem}\label{t:main}
The estimate
\begin{equation}\label{e:wONS-1d}
\bigg\|  \sum_j \nu_j |e^{it\partial_x^2} f_j|^2   \bigg\|_{L^{2,\infty}_tL^\infty_x(\mathbb{R}^{1+1})} \le C \| \nu \|_{\ell^\beta}
\end{equation}
holds for all systems of orthonormal functions $(f_j)_j$ in $L^2(\mathbb{R})$ and $\nu = (\nu_j)_j$ in $\ell^\beta$ if and only if $\beta < 2$.
\end{theorem}
We note that we have been able to obtain the strong-type estimate \eqref{e:ONS-1d} in the range $\beta \le \frac43$; see the remarks at the end of Section \ref{section:Strichartz}. We also remark that our proof of Theorem \ref{t:main} is robust enough to permit generalization to, say, fractional Schr\"odinger equations.  We refrain from stating such results here, and refer the reader forward to Section \ref{section:Strichartz}. Such a generalization is key to our proof of the maximal estimates in Theorem \ref{t:maximal} since we employ an idea due to Kenig--Ponce--Vega \cite{KPV} that has the effect of switching the roles and space and time at the cost of replacing the classical Schr\"odinger propagator with the fractional Schr\"odinger propagator of order $1/2$. 

\begin{organisation}
Section \ref{section:prelims} contains some preliminaries and a more detailed overview of our approach to proving our main results. In Section \ref{section:Strichartz} we prove Theorem \ref{t:main}, and in Section \ref{section:maximal} we prove Theorem \ref{t:maximal} and Corollary \ref{c:pointwise}. Finally, we collect some additional remarks in Section \ref{section:remarks}.
\end{organisation}

\section{Preliminaries and overview} \label{section:prelims}

\subsection{Preliminaries}

We begin by recalling the definition of the Schatten spaces. For $\beta\in[1,\infty)$, $\mathcal{C}^\beta(\mathcal{H})$ is the set of all compact operators $A$ on the Hilbert space $\mathcal{H}$ such that $\| A \|_{\mathcal{C}^\beta} = \| (s_j(A))_j \|_{\ell^\beta} <\infty$, where $(s_j(A))_j$ are the singular values of the operator $A$. Although the case $\beta = \infty$ will not arise in the present work, we recall that this is the space of bounded linear operators on $\mathcal{H}$ with the usual operator norm. In fact, most important for us will be the cases $\beta=2$ and $\beta=4$, in which case explicit computations will be available. Indeed, for Hilbert--Schmidt integral operators of the form
\[
Af(x) = \int_{\mathbb{R}^d} K(x,y) f(y) \, \mathrm{d}y
\]
with $K \in L^2(\mathbb{R}^d \times \mathbb{R}^d)$, the $\mathcal{C}^2$ norm is given by $\|A\|_{\mathcal{C}^2} = \|K\|_{L^2(\mathbb{R}^d \times \mathbb{R}^d)}$. For the case $\beta = 4$, we use the fact that $\| A \|_{\mathcal{C}^4}^2 = \| A^* A\|_{\mathcal{C}^2}$.

Schatten spaces arise naturally via duality when studying Strichartz estimates for orthonormal systems of initial data. For example, the following is a special case of the duality principle of Frank--Sabin \cite[Lemma 3]{frank-sabin-1} (strictly speaking, the result in \cite{frank-sabin-1} is stated for pure Lebesgue spaces).
\begin{proposition} \label{p:duality}
Let $S : L^2(\mathbb{R}) \to L^{q,\infty}_tL^{r}_x(\mathbb{R}^{1+1})$ be a bounded linear operator for some $q>2$ and $r \geq 2$, and let $\beta \geq 1$. Then 
\begin{equation*}
\bigg\|\sum_j\nu_j |Sf_j|^2\bigg\|_{L^{q/2,\infty}_tL^{r/2}_x(\mathbb{R}^{1+1})} \leq C \|\nu\|_{\ell^{\beta}}
\end{equation*}
holds for all orthonormal systems $(f_j)_j$ in $L^2(\mathbb{R})$ and $\nu=(\nu_j)_j$ in $\ell^{\beta}$, if and only if
\[
\|WSS^*\overline{W}\|_{\mathcal{C}^{\beta'}} \leq C \|W\|_{L^{q_0,2}_tL^{r_0}_x}^2
\]
holds for all $W \in L^{q_0,2}_tL^{r_0}_x(\mathbb{R}^{1+1})$, where $\frac{1}{q} + \frac{1}{q_0} = \frac{1}{2}$ and $\frac{1}{r} + \frac{1}{r_0} = \frac{1}{2}$. Here we identify $L^{\infty,\infty}_t$ with $L^\infty_t$.
\end{proposition}

\begin{notation}
In the remainder of the paper, we will use the notation $A \lesssim B$ to mean $A \leq CB$ for an appropriate constant $C$.  We reserve the notation $P$ for the frequency projection operator given by $\widehat{Pf}(\xi) =  \chi(\xi) \widehat{f}(\xi)$, where $\chi \in C^\infty$ is supported on $[-1,1]$ and identically 1 on $[-1/2,1/2]$.

Next, we introduce the propagator $U_\alpha$ for the fractional Schr\"{o}dinger equation of order $\alpha \in \mathbb{R}_{+} \setminus\{1\}$, 
\[
U_{\alpha}f(t,x) = e^{it|\partial_x|^\alpha}f(x) = \frac{1}{2\pi} \int_{\mathbb{R}}e^{i(x \xi+t|\xi|^\alpha)} \widehat{f}(\xi)\, \mathrm{d}\xi
\]
for appropriate functions $f : \mathbb{R} \to \mathbb{C}$. Here,
\[
\widehat{f}(\xi) = \int_{\mathbb{R}} f(x) e^{-i x\xi} \, \mathrm{d}\xi
\]
is the Fourier transform of $f$, and we use the notation $|\partial_x| = (-\partial_x^2)^{1/2}$.
\end{notation}
Although each of our main results stated in the Introduction concern the classical Schr\"odinger propagator $U_2$, as we shall see momentarily, the case $\alpha = 1/2$ will also play an important role.

\subsection{Overview of our proofs}

We shall begin by proving the Strichartz estimates for orthonormal systems contained in Theorem \ref{t:main}. Since the desired estimate \eqref{e:wONS-1d} is scaling invariant, it suffices to prove the analogous estimate with $U_2$ replaced by frequency-localized version $U_2 P$. Thus, in light of Proposition \ref{p:duality}, we will consider estimates on $\| W U_2 P^2 U_2^* \overline{W}\|_{C^{\beta'}}$ with $\beta' > 2$. In order to capitalize on the time decay of the kernel of $U_2 P^2 U_2^*$, we perform an appropriate dyadic decomposition of the operator. By establishing an appropriate range of $\mathcal{C}^2$ and $\mathcal{C}^4$ estimates for each operator arising in this decomposition, we shall we able to obtain the desired estimates on $\| W U_2 P^2 U_2^* \overline{W}\|_{C^{\beta'}}$ via a bilinear interpolation argument inspired by ideas in \cite{KeelTao}.

Next, we shall turn to the proof of Theorem \ref{t:maximal}. Here, we employ a trick due to Kenig--Ponce--Vega \cite{KPV} where they reduced the maximal-in-time estimate \eqref{e:maxi} for $U_2$ to a Strichartz estimate (maximal-in-space estimate) for $U_{1/2}$. More precisely, by an elementary changes of variables, note that
\begin{align}\label{e:change}
2e^{-it\partial_x^2}|\partial_x|^{-\frac14} f(x) & = \frac{1}{2\pi} \int_0^\infty e^{it\eta} |\eta|^{-\frac58} \big( e^{ix|\eta|^{\frac12}}\widehat{f}(|\eta|^{\frac12}) + e^{-ix|\eta|^{\frac12}} \widehat{f}(-|\eta|^{\frac12}) \big)\, \mathrm{d}\eta \\
& =: e^{ix|\partial_t|^{\frac12}}|\partial_t|^{-\frac38} f_+ (t) + e^{-ix|\partial_t|^{\frac12}}|\partial_t|^{-\frac38} f_- (t), \nonumber
\end{align}
where $f_{\pm}$ are given by 
$$
\widehat{f_\pm}(\eta) = |\eta|^{-\frac14} \1_{(0,\infty)}(\eta)\widehat{f}(\pm|\eta|^\frac12).
$$ 
Fortunately, our proof of Theorem \ref{t:main} is sufficiently robust to allow us to obtain, in a straightforward manner, the desired Strichartz estimate for $U_{1/2}$. In fact, we present a somewhat general result in Proposition \ref{p:abstract} which allows us to deduce $L^{4,\infty}_tL^\infty_x$ orthonormal Strichartz estimates for both $U_2$ and $U_{1/2}$. Using the above trick of Kenig--Ponce--Vega we are able to obtain Theorem \ref{t:maximal}; however, an additional step is required to overcome the fact that orthonormal structure is not preserved under the transformation $f \mapsto f_\pm$ (see Lemma \ref{l:symmetric}). 

To show the sharpness of Theorem \ref{t:maximal}, we will employ a semi-classical limit argument and show the failure of the induced estimate by a geometric argument based on the existence of Nikodym sets with zero Lebesgue measure.

\section{Proof of Theorem \ref{t:main}} \label{section:Strichartz}
First, we observe that in order to prove \eqref{e:wONS-1d}, by an elementary rescaling argument, it suffices to prove the frequency localized estimate
\begin{equation} \label{e:U2local}
\bigg\| \sum_{j} \nu_j |U_2P f_j|^2 \bigg\|_{L^{2,\infty}_tL^\infty_x(\mathbb{R}^{1+1})} \lesssim \| \nu\|_{\ell^\beta}.
\end{equation}
Next, we have 
\[
	U_2P^2U_2^*F(t,x) = \int_{\mathbb{R}^{1+1}} K(t-t',x-x') F(t',x')\, \mathrm{d}t'\mathrm{d}x',
\]
for suitable test functions $F:\mathbb{R}^{1+1} \to \mathbb{C}$, where the integral kernel $K$ has the decay property 
$
|K(t,x)| \lesssim (1 + |t|)^{-\frac12}
$
uniformly in $x$. In fact, one can check from a direct computation that 
$$
K(t,x) = \int_{\mathbb{R}} \chi(\xi)^2e^{i(x \xi + t|\xi|^2)}\, \mathrm{d}\xi
$$
and such a decay estimate is a consequence of a stationary phase argument. Our argument for proving \eqref{e:U2local} only uses the above two properties of the operator $U_2$. For this reason, and for use in our forthcoming proof of Theorem \ref{t:maximal}, we consider a more general bounded linear operator $S$ from $L^2(\mathbb{R})$ to $L^{4,\infty}_tL^\infty_x(\mathbb{R}^{1+1})$ such that $SS^*$ is given by
\begin{equation}\label{e:KernelRep}
SS^*F(t,x) = \int_{\mathbb{R}^{1+1}} K(t-t',x-x') F(t',x')\, \mathrm{d}t'\mathrm{d}x'
\end{equation}
on a suitable class of test functions $F:\mathbb{R}^{1+1} \to \mathbb{C}$, where the kernel satisfies
\begin{equation}\label{e:dispersive}
\sup_{x \in \mathbb{R}} |K(t,x)| \lesssim (1 + |t|)^{-\frac12}  \qquad (t \in \mathbb{R}).	
\end{equation}

\begin{proposition}\label{p:abstract}
Suppose $\beta < 2$. Under the assumptions \eqref{e:KernelRep} and \eqref{e:dispersive}, the estimate
$$
\bigg\| \sum_{j} \nu_j |S f_j|^2  \bigg\|_{L^{2,\infty}_tL^\infty_x(\mathbb{R}^{1+1})} \lesssim \| \nu\|_{\ell^\beta} 
$$
holds for all orthonormal systems $(f_j)_j$ in $L^2(\mathbb{R})$ and $\nu=(\nu_j)_j$ in $\ell^\beta$.
\end{proposition} 

Thanks to Proposition \ref{p:duality}, our goal is equivalent to 
\begin{equation}\label{e:goal1}
\| W SS^* \overline{W} \|_{\mathcal{C}^{\beta'}(L^2)} \lesssim \| W \|_{L^{4,2}_tL^2_x}^2
\end{equation}
for all $W \in L^{4,2}_tL^2_x(\mathbb{R}^{1+1})$. To establish this, we first decompose the operator by using the dyadic partition of unity 
\[
\chi+ \sum_{l=1}^\infty \psi(2^{-l} \cdot)\equiv 1, 
\]
where $\psi \in C^\infty_c ( (-2,-\frac12) \cup (\frac12, 2) )$ is a  suitable nonnegative even function. In view of the kernel representation \eqref{e:KernelRep}, we break up the operator $SS^*$ as follows 
\begin{align*}
SS^* F (t,x) & = \int_{\mathbb{R}^{1+1}} K(t-t',x-x')F(t',x')\, \mathrm{d}t' \mathrm{d}x' := T_0 F(t,x)+ \sum_{l\ge 1}T_l F(t,x),
\end{align*}
where 
\begin{align*}
T_0F(t,x) &= \int_{\mathbb{R}^{1+1}} \chi(t-t') K(t-t',x-x')F(t',x')\, \mathrm{d}t'\mathrm{d}x', \\
T_l F(t,x) &= \int_{\mathbb{R}^{1+1}} \psi(2^{-l}(t-t')) K(t-t',x-x')F(t',x')\, \mathrm{d}t'\mathrm{d}x',
\end{align*}
and thus the integral kernel of $T_l$ is given by 
$$
K_l(t-t',x-x') = \psi(2^{-l}(t-t')) K(t-t',x-x').
$$
In order to estimate each term $\| W T_l \overline{W}\|_{\mathcal{C}^{\beta'}(L^2)}$, we use the following.
\begin{lemma} \label{l:eachpiece}
For  $l\ge 1$ we have the following estimates with $C$  independent of $l$: 
\begin{align}
\label{e:Schatten2}
\big\| W_1 T_l W_2 \big\|_{\mathcal{C}^{2}(L^2)} & \le C \| W_1 \|_{L^4_tL^2_x} \| W_2 \|_{L^4_tL^2_x}, 
\\
\label{e:Schatten4}
\big\| W_1 T_l W_2 \big\|_{\mathcal{C}^{4}(L^2)} &\le C 2^{(\frac12-\frac1{p_1}-\frac1{p_2})l} \| W_1 \|_{L^{p_1}_tL^2_x} \| W_2 \|_{L^{p_2}_tL^2_x}
\end{align}
provided that $p_1$ and $p_2$ satisfy  $(\frac1{p_1},\frac1{p_2}) \in [0,\frac12]^2$  and 
$$
\frac1{p_1}+\frac1{p_2} \ge \frac14,\;\;\; \frac1{p_1} - \frac1{2p_2} \le \frac14,\;\;\; \frac1{p_2} -\frac1{2p_1}  \le \frac14.
$$

\end{lemma}

\begin{proof}[Proof of Lemma \ref{l:eachpiece}]
For \eqref{e:Schatten2}, since
\[
|K_l(t-t',x-x')| \lesssim \psi(2^{-l}(t-t')) |t-t'|^{-1/2}
\]
follows from \eqref{e:dispersive}, we may use Young's convolution inequality to obtain
\begin{align*}
\big\| W_1 T_l W_2 \big\|_{\mathcal{C}^{2}(L^2)}^2 & = \int |W_1(t,x)|^2 |K_l(t-t',x-x')|^2 |W_2(t',x')|^2 \, \mathrm{d}t\mathrm{d}t'\mathrm{d}x\mathrm{d}x' \\
& \lesssim \| W_1 \|_{L^4_tL^2_x}^2 \| W_2 \|_{L^4_tL^2_x}^2 
\end{align*}
uniformly in $l$.

For \eqref{e:Schatten4}, first note that 
\[
\big\| W_1 T_l W_2 \big\|_{\mathcal{C}^4(L^2)}^4 
= 
\big\| \overline{W_2} T_l^* |W_1|^2 T_l W_2 \big\|_{\mathcal{C}^2(L^2)}^2,
\]
and after a few lines of computation we see that 
\[
\overline{W_2} T_l^* |W_1|^2 T_l W_2 [F](t,x) = \int_{\mathbb{R}^{1+1}} \mathcal{K}_l(t,t'',x,x'') F(t'',x'')\, \mathrm{d}t''\mathrm{d}x'',
\]
where the integral kernel $\mathcal{K}_l(t,t'',x,x'') $ is given by 
\[
\overline{W_2(t,x)} \int_{\mathbb{R}^{1+1}} K_l(t-t',x-x')|W_1(t',x')|^2 K_l(t'-t'',x'-x'') \, \mathrm{d}t'\mathrm{d}x' \cdot W_2(t'',x''). 
\]
Thanks to \eqref{e:dispersive} (and relabelling $t = t_1$ and $t'' = t_4$) we see 
\begin{align*}
\big\| W_1 T_l W_2 \big\|_{\mathcal{C}^4(L^2)}^4 
& = 
\int_{\mathbb{R}^{1+1}} \int_{\mathbb{R}^{1+1}}   |\mathcal{K}_l(t_1,t_4,x,x'')|^2\, \mathrm{d}t_1 \mathrm{d}x \mathrm{d}t_4  \mathrm{d}x'' \\
& \lesssim 
\int_{\mathbb{R}} \int_{\mathbb{R}} \|W_2(t_1,\cdot)\|_{L^2_x}^2 L_l(t_1,t_4)^2 \|W_2(t_4,\cdot)\|_{L^2_x}^2 \, \mathrm{d}t_1 \mathrm{d}t_4, 
\end{align*}
where 
$$
L_l(t_1,t_4) = 
\int_{\mathbb{R}} 2^{-l/2} \psi(2^{-l}(t_1-t')) \|W_1(t',\cdot)\|_{L^2_x}^2 2^{-l/2}\psi(2^{-l}(t'-t_4)) \, \mathrm{d}t'.
$$
By multiplying out the square of $L_l(t_1,t_4)$, we have
\begin{equation} \label{e:(2)main}
\big\| W_1 T_l W_2 \big\|_{\mathcal{C}^4(L^2)}^4 \lesssim \Lambda_l (h_2,h_1,h_1,h_2),
\end{equation}
where  $h_k(t) = \|W_k(t,\cdot)\|_{L^2_x}^2$, $k=1,2$, and  $\Lambda_l$ is the $4$-linear form given by
\begin{align*}
\Lambda_l(g_1,g_2,g_3,g_4) 
& := 
2^{-2l} \int_{\mathbb{R}^4} \psi(2^{-l}(t_1-t_2)) \psi(2^{-l}(t_1-t_3)) \\
& \quad \times \psi(2^{-l}(t_2-t_4)) \psi(2^{-l}(t_3-t_4)) 
 \prod_{i=1}^4 g_i(t_i)\,\mathrm{d}t_i
\end{align*}

We shall prove the desired estimates \eqref{e:Schatten4} at the three points $(p_1,p_2) = (2,2), (4,\infty)$, and $(\infty,4)$, and then employ multilinear interpolation to deduce the estimates in the claimed region. The estimate at $({p_1},{p_2}) = (2,2)$ follows immediately from \eqref{e:(2)main} and the fact that $\|\psi\|_{L^\infty} \lesssim 1$. Thus, by symmetry, it suffices to prove \eqref{e:Schatten4} at $({p_1},{p_2}) = (4,\infty)$.
 
From $\|\psi\|_{L^\infty} \lesssim 1$ and $\|\psi(2^{-l}\cdot)\|_{L^1} \sim 2^l$, we see 
\begin{align*}
|\Lambda_l(g_1,g_2,g_3,g_4)| 
& \lesssim 2^{-2l} \|g_1\|_{L^\infty} \|g_3\|_{L^\infty} \|g_4\|_{L^\infty} \\
&\times\int_{\mathbb{R}^4} \psi(2^{-l}(t_1-t_2)) \psi(2^{-l}(t_1-t_3)) \psi(2^{-l}(t_2-t_4)) g_2(t_2)\, \prod_{i=1}^4\mathrm{d}t_i\\
& \sim 
2^l\|g_1\|_{L^\infty}\|g_2\|_{L^1}\|g_3\|_{L^\infty}\|g_4\|_{L^\infty}.
\end{align*}
By symmetry, we also have 
$$
|\Lambda_l(g_1,g_2,g_3,g_4)| \lesssim 2^l\|g_1\|_{L^\infty}\|g_2\|_{L^\infty}\|g_3\|_{L^1}\|g_4\|_{L^\infty},
$$
and so it follows from interpolation between these two bounds that 
$$
|\Lambda_l(g_1,g_2,g_3,g_4)| \lesssim 2^l\|g_1\|_{L^\infty}\|g_2\|_{L^2}\|g_3\|_{L^2}\|g_4\|_{L^\infty}.
$$
In particular, from \eqref{e:(2)main} this implies 
$$
\big\| W_1 T_l W_2 \big\|_{\mathcal{C}^4(L^2)}^4 \lesssim 2^l\|h_1\|_{L^2}^2 \|h_2\|_{L^\infty}^2 = 2^l\|W_1\|_{L^4_tL^2_x}^4 \|W_2\|_{L^\infty_tL^2_x}^4
$$
which gives \eqref{e:Schatten4} at $({p_1},{p_2}) = (4,\infty)$. 
\end{proof}

\begin{proof}[Proof of Proposition \ref{p:abstract}] By interpolating between \eqref{e:Schatten2} and \eqref{e:Schatten4}, for each $2< \beta' \le 4$, there exists $\delta(\beta')>0$ such that $\lim_{\beta'\to 2} \delta(\beta') = 0$ and that the estimate
\begin{align}
\label{tlbeta}
\big\| W_1 T_l W_2 \big\|_{\mathcal{C}^{\beta'}(L^2)} 
\lesssim 2^{(\frac12-\frac1{p_1}-\frac1{p_2})l} \| W_1 \|_{L^{p_1}_tL^2_x} \| W_2 \|_{L^{p_2}_tL^2_x} 
\end{align}
holds for all $p_1,p_2$ satisfying $|(\frac1{p_1},\frac1{p_2}) - (\frac14,\frac14)| \le \delta(\beta')$. 

Now fix $\beta_* < 2$ such that $\beta_*' \in (2,4)$ and define the bilinear operator $\mathcal{T}$ by
$$
\mathcal{T}(W_1,W_2) := ( W_1 T_l W_2 )_{l \geq 1}.
$$
Then \eqref{tlbeta} shows $\mathcal{T} : L^{p_1}_tL^2_x \times L^{p_2}_tL^2_x \to \ell^\infty_{\mu(p_1,p_2)} (\mathcal{C}^{\beta_*'}(L^2))$ is bounded for all $(\frac1{p_1},\frac1{p_2})$ in $\delta(\beta_*')$-neighborhood of $(\frac14,\frac14)$. Here, $\mu(p_1,p_2) = \frac1{p_1}+\frac1{p_2}-\frac12$ and, for a general Banach space $X$ and sequence $(g_l)_l \subset X$, the norm is given by  
\[
\| (g_l)_l \|_{\ell^p_\mu(X)} = \bigg( \sum_{l} 2^{p\mu l} \| g_l \|_X^p \bigg)^{1/p}
\] 
for $p < \infty$, and $ \| (g_l)_l \|_{\ell^\infty_\mu(X)} = \sup_l 2^{\mu l} \|g_l\|_X$. Therefore, a bilinear real interpolation argument (see \cite[Exercise 5, Page 76]{BerghLofstrom}) reveals that $\mathcal{T} : L^{4,2}_tL^2_x \times L^{4,2}_tL^2_x \to \ell^1(\mathcal{C}^{\beta_*'}(L^2))$ is bounded, which in particular means 
$$
\sum_{l\ge1} \big\| W T_l \overline{W} \big\|_{\mathcal{C}^{\beta_*'}(L^2)} \lesssim \| W \|_{L^{4,2}_t L^2_x}^2.
$$

Finally, we consider the operator $W_1T_0 W_2$. Note that, in a similar manner to the proof of Lemma \ref{l:eachpiece}(1), we have
\begin{align*}
\big\| W T_0 \overline{W} \big\|_{\mathcal{C}^{2}(L^2)}^2 & = \int |W(t,x)|^2 \chi^2(t-t') |K(t-t',x-x')|^2 |W(t',x')|^2 \, \mathrm{d}t\mathrm{d}t'\mathrm{d}x\mathrm{d}x' \\
& \lesssim \int \|W(t,\cdot)\|_{L^2_x}^2  \chi^2(t-t')  \|W(t',\cdot)\|^2 \, \mathrm{d}t\mathrm{d}t' \\
& \lesssim \| W \|_{L^4_tL^2_x}^4
\end{align*}
which certainly implies
\[
\| W T_0 \overline{W} \|_{\mathcal{C}^{\beta_*'}(L^2)} \lesssim \| W \|_{L^{4,2}_t L^2_x}^2
\]
via embeddings. Hence, by the triangle inequality, we have 
\begin{equation} \label{e:break}
\big\| W SS^* \overline{W} \big\|_{\mathcal{C}^{\beta_*'}(L^2)} \le \big\| W T_0 \overline{W} \big\|_{\mathcal{C}^{\beta_*'}(L^2)} + \sum_{l\ge1} \big\| W T_l \overline{W} \big\|_{\mathcal{C}^{\beta_*'}(L^2)} \lesssim  \| W \|_{L^{4,2}_t L^2_x}^2
\end{equation}
which yields our goal \eqref{e:goal1} for all $\beta_* \in (2,4)$. This suffices to prove Proposition \ref{p:abstract}.
\end{proof}

\begin{proof}[Proof of Theorem \ref{t:main}]
As indicated at the start of Section \ref{section:Strichartz}, for the sufficiency part of Theorem \ref{t:main}, we obtain \eqref{e:U2local} as an immediate consequence of Proposition \ref{p:abstract} with $S = U_2P$, and a standard rescaling argument to remove the frequency cut-off $P$ yields \eqref{e:wONS-1d}. 

To see the necessity of $\beta<2$, it is enough to show the failure of \eqref{e:wONS-1d} with $\beta=2$ thanks to the inclusion relation between Schatten spaces. Suppose \eqref{e:wONS-1d} holds true with $\beta = 2$, then from Proposition \ref{p:duality}, we also have 
$$
\| W_1 U_2U_2^* W_2 \|_{\mathcal{C}^2(L^2)} \lesssim \| W_1 \|_{L^{4,2}_tL^2_x} \| W_2 \|_{L^{4,2}_tL^2_x}. 
$$
On the other hand, the kernel of $UU^*$ is given by $C |t-t'|^{-\frac12} e^{-\frac{|x-x'|^2}{4i(t-t')}}$ and therefore
$$
\| W_1 U_2U_2^* W_2 \|_{\mathcal{C}^2(L^2)}^2 = C \int_{\mathbb{R}^{1+1}}\int_{\mathbb{R}^{1+1}} |W_1(t,x)|^2 |t-t'|^{-1} |W_2(t',x')|^2\, \mathrm{d}t\mathrm{d}x\mathrm{d}t'\mathrm{d}x' .
$$
Choosing $W_1 = W_2 = \1_{[-1,1]^2}$, we get a contradiction because $\| W_1 U_2U_2^* W_2 \|_{\mathcal{C}^2(L^2)}=\infty$.  This  establishes the necessity of the condition $\beta < 2$ for \eqref{e:wONS-1d}.
\end{proof}

\begin{remarks}
\begin{enumerate}
[leftmargin=0cm, labelsep=4pt, parsep=7pt, itemindent=15pt]
\item[$(\rm I)$]
Ideas in the proof of Proposition \ref{p:abstract} yield the bound
\begin{equation}\label{e:Strong}
\bigg\|  \sum_j \nu_j |e^{it\partial_x^2} f_j|^2 \bigg\|_{L^{2,\frac43}_tL^\infty_x(\mathbb{R}^{1+1})} \lesssim \| \nu \|_{\ell^\frac43}
\end{equation}
for orthonormal systems $(f_j)_j$ in $L^2(\mathbb{R})$ and $\nu = (\nu_j)_j$ in $\ell^{\frac43}$, which, in particular, is a Lorentz-space improvement of the strong-type estimate
\begin{equation*}
\bigg\|  \sum_j \nu_j |e^{it\partial_x^2} f_j|^2 \bigg\|_{L^{2}_tL^\infty_x(\mathbb{R}^{1+1})} \lesssim \| \nu \|_{\ell^\frac43}.
\end{equation*}
Indeed, if one computes $\| W U_2U_2^* \overline{W}\|_{\mathcal{C}^4(L^2)}$ directly as we did in the above, then we obtain
\begin{equation}\label{e:MultiForm}
\big\| W U_2U_2^* \overline{W} \big\|_{\mathcal{C}^4}^4 
\lesssim 
\int_{\mathbb{R}^4} |t_1-t_2|^{-\frac12} |t_1-t_3|^{-\frac12} |t_2-t_4|^{-\frac12} |t_3-t_4|^{-\frac12} \prod_{i=1}^4 h(t_i)\, \mathrm{d}t_i,
\end{equation}
where $h(t) = \| W(t,\cdot) \|_{L^2_x}^2$. We regard the right-hand side of \eqref{e:MultiForm} as an 8-linear rank-one Brascamp--Lieb form and we may use Barthe's characterisation in \cite{Barthe} of the associated Brascamp--Lieb polytope in the rank-one case and Christ's observations in \cite{Perry} on extending classical Brascamp--Lieb estimates to Lorentz spaces to conclude
\begin{equation} \label{e:StrongSchatten}
\big\| W U_2U_2^* \overline{W} \big\|_{\mathcal{C}^4(L^2)} \lesssim \| W \|_{L^{4,8}_tL^2_x}^2,
\end{equation}   
or equivalently, \eqref{e:Strong}. We refer the reader to \cite{BCCT_GAFA,BL,Brown,BLNS} for further details regarding the Brascamp--Lieb inequality and its Lorentz space refinement.

\item[$(\rm I\!I)$]
If one can appropriately exploit the orthogonality of the $T_l$, rather than the application of the triangle inequality in \eqref{e:break}, it seems possible to upgrade \eqref{e:wONS-1d} to a strong-type estimate. For instance, it seems reasonable to expect that for all $\beta \in [1,2]$ we have 
\begin{equation}\label{e:orthogonality}
\big\| W U P^2 U^* \overline{W} \big\|_{\mathcal{C}^{\beta'}(L^2)} \le \big\| W T_0 \overline{W} \big\|_{\mathcal{C}^{\beta'}(L^2)} + \bigg( \sum_{l\ge1} \big\| W T_l \overline{W} \big\|_{\mathcal{C}^{\beta'}(L^2)}^\beta \bigg)^{1/\beta}. 
\end{equation}
Indeed, if $\beta=2$ it is easy to  see that \eqref{e:orthogonality} holds in this case. Also, \eqref{e:orthogonality} for $\beta =1$ is an easy consequence of the triangle inequality. However, it is not clear to us how to interpolate these two estimates.  

Under the assumption that \eqref{e:orthogonality} holds for all $\beta \in [1,2]$, it follows from our argument to show \eqref{e:wONS-1d} that 
\begin{equation}\label{e:Improved}
\big\| W UU^* \overline{W} \big\|_{\mathcal{C}^{\beta'}(L^2)} \lesssim \| W \|_{L^{4,2\beta}_t,L^2_x}^2,
\end{equation}
for $\beta<2$ arbitrary close to 2. On the other hand, one has a Lorentz improvement if $\beta' = 4$ as in \eqref{e:StrongSchatten}. Interpolating \eqref{e:Improved} and \eqref{e:StrongSchatten}, one would obtain the desired strong-type estimate for any $\beta<2$. 

\item[$(\rm I\!I\!I)$]
Assuming the more general decay hypothesis
\begin{equation*}
\sup_{x \in \mathbb{R}} |K(t,x)| \lesssim (1 + |t|)^{-\sigma}	 \qquad (t \in \mathbb{R})
\end{equation*}
for some $\sigma > 0$, one may easily generalize our argument in the proof of Proposition \ref{p:abstract} to obtain
$$
\bigg\|  \sum_{j} \nu_j |S f_j|^2  \bigg\|_{L^{q/2,\infty}_tL^\infty_x(\mathbb{R}^{1+1})} \lesssim \| \nu\|_{\ell^\beta}
$$
for orthonormal systems $(f_j)_j$ in $L^2(\mathbb{R})$ and $\nu = (\nu_j)_j$ in $\ell^{\beta}$, where $q = \max\{\frac{2}{\sigma},4\}$ and $\beta<2$. It is also clear from an inspection of the proof that the domain of the spatial variable may be generalized.
\end{enumerate}
\end{remarks}

\section{Proofs of Theorem \ref{t:maximal} and Corollary \ref{c:pointwise}} \label{section:maximal}

Recalling the identity \eqref{e:change}, our first step in establishing Theorem \ref{t:maximal} is to observe the following analogue of Theorem \ref{t:main}.
\begin{theorem}\label{t:fractional-end}
Suppose $\beta < 2$. Then the estimate
\[
\bigg\|  \sum_{j}\nu_j |U_{\frac12}|\partial_x|^{-\frac38} f_j |^2 \bigg\|_{L^{2,\infty}_tL^\infty_x(\mathbb{R}^{1+1})} \lesssim \| \nu \|_{\ell^\beta}
\]
holds for all systems of orthonormal functions $(f_j)_j$ in $L^2(\mathbb{R})$ and $\nu=(\nu_j)_j$ in $\ell^\beta$. 
\end{theorem}
\begin{proof}
We invoke Proposition \ref{p:abstract} with $S = U_\frac12 |\partial_x|^{-\frac38} P$. It is clear that \eqref{e:KernelRep} holds with
\[
K(x,t) = \int_{\mathbb{R}} \chi(\xi)^2  |\xi|^{-3/4} e^{i(x \xi + t|\xi|^{1/2})}\, \mathrm{d}\xi,
\]
and the desired decay estimate \eqref{e:dispersive} holds thanks to work of Kenig--Ponce--Vega \cite[Lemma 2.7]{KPV}.
\end{proof}

\subsection{Proof of Theorem \ref{t:maximal} (Sufficiency part)}
A minor snag which arises when using \eqref{e:change} is that the transformation $f\mapsto f_{\pm}$ does not preserve the orthogonal structure and the orthonormality of $(f_j)_j$ does not always ensure the orthonormality of $(f_{j,+})_j$ and $(f_{j,-})_j$. To recover this, we introduce the reflection operator $R$ given by
\[
R\varphi(t,x) := \varphi(-t,-x)
\] 
and establish the following.
\begin{lemma}\label{l:symmetric}
\begin{enumerate}
[leftmargin=0cm, labelsep=4pt, parsep=7pt, itemindent=15pt]
\item[$\rm(I)$]
For each $f\in L^2$, 
\begin{equation}\label{e:symmetric}
\sqrt{2}(1\pm R) U_{2}|\partial_x|^{-\frac14}f(t,x) = e^{ix|\partial_t|^{\frac12}}|\partial_t|^{-\frac38}f_{\pm}(t) + e^{-ix|\partial_t|^{\frac12}}|\partial_t|^{-\frac38} f^*_{\pm}(t),
\end{equation}
where \begin{align*}
\widehat{f_{\pm}}(\eta) &= \frac{1}{\sqrt{2}} |\eta|^{-\frac14} \big( \1_{(0,\infty)}(\eta) \widehat{f}(|\eta|^{\frac12}) \pm \1_{(-\infty,0)}(\eta) \widehat{f}(-|\eta|^{\frac12})\big), \\
\widehat{f_{\pm}^*}(\eta) &= \frac{1}{\sqrt{2}} |\eta|^{-\frac14} \big( \1_{(0,\infty)}(\eta) \widehat{f}(-|\eta|^{\frac12}) \pm \1_{(-\infty,0)}(\eta) \widehat{f}(|\eta|^{\frac12}) \big).
\end{align*}
\item[$\rm(I\!I)$]
Suppose $(f_j)_j$ is a orthonormal system in $L^2$. Then each of the families $(f_{j,+})_j$, $(f_{j,-})_j$, $(f_{j,+}^*)_j$ and $(f_{j,-}^*)_j$ is orthonormal in $L^2$. 
\end{enumerate}
\end{lemma}

\begin{proof}
First we show \eqref{e:symmetric}. 
Following the idea in \eqref{e:change}, we have
\begin{align*}
&2R U_2|\partial_x|^{-\frac14}f(t,x) = 2 e^{it\partial_x^{2}}|\partial_x|^{-\frac14} [f(-\cdot)](x) \\
=&
\frac{1}{2\pi} \int_0^\infty e^{ix|\eta|^{\frac12}}e^{-it\eta} |\eta|^{-\frac58} \widehat{f}(-|\eta|^{\frac12})\, \mathrm{d}\eta + \frac{1}{2\pi}\int_0^\infty e^{-ix|\eta|^{\frac12}}e^{-it\eta} |\eta|^{-\frac58} \widehat{f}(|\eta|^{\frac12})\, \mathrm{d}\eta \\
=&
\frac{1}{2\pi} \int_{-\infty}^0 e^{ix|\eta|^{\frac12}}e^{it\eta} |\eta|^{-\frac58} \widehat{f}(-|\eta|^{\frac12})\, \mathrm{d}\eta + \frac{1}{2\pi}\int_{-\infty}^0 e^{-ix|\eta|^{\frac12}}e^{it\eta} |\eta|^{-\frac58} \widehat{f}(|\eta|^{\frac12})\, \mathrm{d}\eta,
\end{align*}
from which we obtain \eqref{e:symmetric}.

Next, let us see the orthonormality. Using Parseval's identity, 
\begin{align*}
&4\pi \langle f_{j,+}, f_{k,+} \rangle_{L^2} \\
=& 
\int_{\mathbb{R}} |\eta|^{-\frac12} \big( \1_{(0,\infty)}(\eta)  {\widehat{f_j}(|\eta|^{\frac12})} \overline{\widehat{f_k}(|\eta|^{\frac12})} + \1_{(-\infty,0)}(\eta) {\widehat{f_j}(-|\eta|^{\frac12})} \overline{\widehat{f_k}(-|\eta|^{\frac12})}  \big)\, \mathrm{d}\eta \\
=&
\int_0^\infty|\eta|^{-\frac12} \big( {\widehat{f_j}(|\eta|^{\frac12})}  \overline{\widehat{f_k}(|\eta|^{\frac12})} +  {\widehat{f_j}(-|\eta|^{\frac12})} \overline{\widehat{f_k}(-|\eta|^{\frac12})}\big)\, \mathrm{d}\eta, 
\end{align*}
where we performed a simple change of the variable for the second term in the last equality. 
Then the change of variable $\xi^2 =\eta$ yields
\begin{align*}
4\pi \langle f_{j,+}, f_{k,+} \rangle_{L^2}
&=
2\int_0^\infty \big(  {\widehat{f_j}(\xi)} \overline{\widehat{f_k}(\xi)} +  {\widehat{f_j}(-\xi)} \overline{\widehat{f_k}(-\xi)} \big)\, \mathrm{d}\xi \\
& = 2 \int_{\mathbb{R}} \widehat{f_j}(\xi) \overline{{\widehat{f_k}(\xi)}}\, \mathrm{d}\xi = 4\pi \langle f_{j}, f_{k} \rangle_{L^2},
\end{align*}
and hence $\langle f_{j,+}, f_{k,+} \rangle_{L^2} = \langle f_{j}, f_{k} \rangle_{L^2}$. By a very similar calculation we also have $\langle f_{j,-}, f_{k,-} \rangle_{L^2} = \langle f_{j}, f_{k} \rangle_{L^2}$. Finally, since $f_{\pm}^* = f(-\cdot)_{\pm}$, we may deduce the corresponding identities for $(f_{j,+}^*)_j$ and $(f_{j,-}^*)_j$.
\end{proof}

\begin{proof}[Proof of the sufficiency part of Theorem \ref{t:maximal}]
By writing
$
2 U_2 = (1+R) U_2 + (1-R) U_2,
$
and applying Lemma \ref{l:symmetric} and the triangle inequality, we have
\begin{align*}
\bigg\| \sum_j \nu_j | U_2|\partial_x|^{-\frac14}f_j|^2   \bigg\|_{L^{2,\infty}_xL^\infty_t} \lesssim   N_++N_+^* + N_-+N_-^*
\end{align*}
where
\begin{align*} 
N_\pm & := \bigg\|  \sum_j \nu_j |e^{ix|\partial_t|^{\frac12}}|\partial_t|^{-\frac38}f_{j,\pm} |^2  \bigg\|_{L^{2,\infty}_xL^\infty_t}, \\
N_\pm^*  & := \bigg\| \sum_j \nu_j |e^{-ix|\partial_t|^{\frac12}}|\partial_t|^{-\frac38}f_{j,\pm}^* |^2\bigg\|_{L^{2,\infty}_xL^\infty_t}.
\end{align*}
Hence, applying Theorem \ref{t:fractional-end}, for $\beta < 2$ we get 
\[
\bigg\| \sum_j \nu_j | U_2|\partial_x|^{-\frac14}f_j|^2 \bigg\|_{L^{2,\infty}_xL^\infty_t} \lesssim \|\nu\|_{\ell^\beta}. \qedhere
\]
\end{proof}

\subsection{Proof of Theorem \ref{t:maximal} (Sharpness)}

Our goal is to show that the estimate 
$$
\bigg\| \sum_j \nu_j |e^{it\partial_x^2} |\partial_x|^{-\frac14}f_j|^2 \bigg\|_{L^{2,\infty}_xL^\infty_t(\mathbb{R}^{1+1})} \lesssim \| \nu \|_{\ell^{2,1}}
$$
for systems of orthonormal functions $(f_j)_j$ in $L^2(\mathbb{R})$ and $\nu = (\nu_j)_j$ in $\ell^{2,1}$ is false. If this estimate were true, by a semi-classical limiting argument, we may induce the following maximal estimate for the (weighted) velocity average of the kinetic transport equation 
\begin{equation}\label{e:velocity}
\bigg\| \int_{\mathbb{R}}f(x-tv,v)\, \frac{\mathrm{d}v}{|v|^{\frac12}} \bigg\|_{L^{2,\infty}_x L^\infty_t(\mathbb{R}^{1+1})} \lesssim \|f\|_{L^{2,1}_{x,v}}
\end{equation}
for any $f\in L^{2,1}_{x,v}(\mathbb{R}^{1+1})$. We refer the reader to \cite{BHLNS,sabin} for further details of such a limiting procedure.

\begin{proof}[Proof that \eqref{e:velocity} fails]
Suppose $\mathcal{N}\subset [-10,10]^2$ has Lebesgue measure zero and contains a unit line segment whose angle from the vertical line is at most $\frac{\pi}4$ through every point of $\{ (x,0):x\in[-1,1] \}$. Sets with the latter geometric property are often referred to as Nikodym sets and the existence of such sets with Lebesgue measure zero goes back to \cite{Nikodym} (see also \cite{Wisewell} for further discussion and an explicit construction).

Let us denote the $\delta$-neighbourhood of $\mathcal{N}$ by $\mathcal{N}_\delta$ for each $\delta>0$; we shall test \eqref{e:velocity} on the characteristic function $f = \1_{\mathcal{N}_\delta}$. Thanks to the geometric property of the Nikodym set $N$, for any $x\in [-1,1]$ there exist $t(x) \in \mathbb{R}$ and a unit interval $I(x)$ such that 
\[
\{ (x-v (-t(x),1)' : v\in I(x) \}
\] 
is contained in $\mathcal{N}$. Here, we use the notation $\omega' := |\omega|^{-1} \omega$. Note that, $t(x) \in [-1,1]$, thanks to the restriction of the angle to the vertical, and furthermore we have $I(x) \subset [-20,20]$ since $\mathcal{N}  \subset [-10,10]^2$. Thus, for any $x \in \mathbb{R}$, the above yields 
$$
\bigg\| \int_{\mathbb{R}}1_{\mathcal{N}_\delta}(x-tv,v)\, \frac{\mathrm{d}v}{|v|^{\frac12}} \bigg\|_{L^\infty_t}  \geq \1_{[-1,1]}(x) \int_{\mathbb{R}}1_{\mathcal{N}_\delta}(x-t(x)v,v)\, \frac{\mathrm{d}v}{|v|^{1/2}} 
\gtrsim 1,
$$
and hence
$$
\bigg\| \int_{\mathbb{R}}1_{\mathcal{N}_\delta}(x-tv,v)\, \frac{\mathrm{d}v}{|v|^{\frac12}} \bigg\|_{L^{2,\infty}_x L^\infty_t}  \gtrsim 1,
$$
with implicit constants uniform in $\delta>0$. On the other hand, since $\mathcal{N}$ has zero Lebesgue measure, we have $\|1_{\mathcal{N}_\delta}\|_{L^{2,1}_{x,v}} \sim |\mathcal{N}_\delta|^{\frac12} \to 0$ as $\delta\to0$. This establishes that \eqref{e:velocity} is false.  
\end{proof}

Whilst it was rather easier to establish the failure of the case $\beta = 2$ in Theorem \ref{t:main} by using duality and explicit computations using the Hilbert--Schmidt norm, it seems unclear how to proceed along similar lines for the necessity part of Theorem \ref{t:maximal}. More precisely, in order to see that \eqref{e:max} fails with $\beta = 2$, by duality it suffices to show the failure of the estimate 
\begin{equation*}
\big\| W_1 |\partial_x|^{-\frac14}U_2U_2^* |\partial_x|^{-\frac14} W_2 \big\|_{\mathcal{C}^2(L^2)} \lesssim \| W_1 \|_{L^{4,2}_xL^2_t} \| W_2 \|_{L^{4,2}_xL^2_t}. 
\end{equation*}
Thanks to the presence of the derivatives $|\partial_x|^{-\frac14}$, however, as far as we aware, it does not seem easy to have a convenient formula for the integral kernel of $|\partial_x|^{-\frac14} U_2U_2^*|\partial_x|^{-\frac14}$. The alternative approach we took in the above using a semi-classical limiting argument circumvents this issue and moreover allows us to show the failure of the restricted weak-type estimate.

\subsection{Proof of Corollary \ref{c:pointwise}}
For a given Hilbert space $\mathcal{H}$ and a unit vector $g\in \mathcal{H}$, we define $\Pi_g : \mathcal{H} \to \mathcal{H}$ to be the orthogonal projection onto the span of $g$ given by $\Pi_g\phi := \langle \phi, g \rangle g$. Note that for any compact operator $\gamma_0$ on $\mathcal{H}$, in particular $\gamma_0 \in \mathcal{C}^\beta(\mathcal{H})$, $\beta<2$, one can find $(\nu_j)_j$ and orthonormal system $(g_j)_j$ in $\mathcal{H}$ such that $\gamma_0 = \sum_j \nu_j \Pi_{g_j}$ thanks to the singular value decomposition. 

For $\gamma_0 \in \mathcal{C}(\dot{H}^\frac{1}{4}(\mathbb{R}))$ and its evolution $\gamma(t) = e^{-it\partial_x^2} \gamma_0 e^{it\partial_x^2}$ under \eqref{e:Hartree-free}, first we clarify the meaning of the density functions $\rho_{\gamma_0}$ and $\rho_{\gamma(t)}$. In the finite-rank case $\gamma_0 = \sum_{j=1}^N \nu_j \Pi_{g_j}$, the integral kernel is given by
\[
(x,y) \mapsto \sum_{j=1}^N \nu_j g_j(x) \overline{g_j(y)} 
\]
and thus we have
\[
\rho_{\gamma_0} (x) = \sum_{j=1}^N \nu_j |g_j(x)|^2.
\]
In the infinite-rank case, some care is required and we proceed via Lieb's generalization of the Sobolev inequality 
\begin{equation} \label{e:Lieb_Sobolev}
\bigg\| \sum_j \nu_j ||\partial_x|^{-\frac14} f_j|^2 \bigg\|_{L^2(\mathbb{R})} \lesssim \|\nu\|_{\ell^1}^{\frac{1}{2}} \|\nu\|_{\ell^\infty}^{\frac{1}{2}}
\end{equation}
for orthonormal systems $(f_j)_j$ in $L^2(\mathbb{R})$ and coefficients $\nu = (\nu_j)_j$ in $\ell^1 \cap \ell^\infty$ (see \cite{Lieb-Sobolev}). We may replace the right-hand side of \eqref{e:Lieb_Sobolev} by $\| \nu \|_{\ell^{2,1}}$ (using, for example, \cite[Ch. 5, Theorem 3.13]{SteinWeiss}) and, in view of the inclusion\footnote{
If $\beta < 2$ and $(\nu_j^*)_j$ is the sequence $(|\nu_j|)_j$ permuted in a decreasing order, we have
$
\| \nu \|_{\beta'} \lesssim (\sum_{j \geq 1} (\nu_j^*)^{\beta'} j^{\beta'/2} \cdot j^{-\beta'/2})^{1/\beta'} \lesssim  \sup_{j \geq 1} j^{1/2} \nu_j^* = \|\nu\|_{\ell^{2,\infty}}
$
and therefore, by duality, $\ell^{\beta} \subseteq \ell^{2,1}$. 
%
} 
$\ell^\beta \subseteq \ell^{2,1}$ for any $\beta <2$, we have 
\begin{equation}\label{e:lieb-sobolev}
\bigg\| \sum_j \nu_j ||\partial_x|^{-\frac14} f_j|^2 \bigg\|_{L^2(\mathbb{R})} \lesssim \| \nu \|_{\ell^\beta}
\end{equation}
for orthonormal systems $(f_j)_j$ in $L^2(\mathbb{R})$, $\nu = (\nu_j)_j$ in $\ell^{\beta}$, and $\beta < 2$.

We now fix $\beta<2$ and approximate $\gamma_0 = \sum_{j=1}^\infty \nu_j \Pi_{g_j} \in \mathcal{C}^\beta(\dot{H}^{\frac14}(\mathbb{R}))$, for $\nu \in \ell^\beta$ and orthonormal vectors $g_j \in \dot{H}^\frac{1}{4}(\mathbb{R})$, by the sequence of finite-rank operators $(\gamma_0^N)_{N\ge1}$ given by $\gamma_0^N = \sum_{j=1}^N \nu_j \Pi_{g_j}$. For $M > N$, we obtain 
\[
\| \rho_{\gamma_0^N} - \rho_{\gamma_0^M} \|_{2} = \bigg\| \sum_{j=N+1}^M \nu_j |g_j|^2 \bigg\|_{2} \lesssim \bigg( \sum_{j=N+1}^M |\nu_j|^\beta \bigg)^{\frac1\beta} 
\]
from \eqref{e:lieb-sobolev}, and therefore $(\rho_{\gamma_0^N})$ is a Cauchy sequence in $L^2(\mathbb{R})$. Thus, we define $\rho_{\gamma_0} = \sum_{j=1}^\infty \nu_j |g_j|^2 \in L^2(\mathbb{R})$ as the limit of $(\rho_{\gamma_0^N})$ in $L^2(\mathbb{R})$. Since orthonormality of $(f_j)_j$ is preserved under the action of $e^{it\partial_x^2}$ for each $t \in \mathbb{R}$, we may repeat the above to define the density function $\rho_{\gamma(t)} = \sum_{j=1}^\infty \nu_j|e^{it\partial_x^2} g_j |^2 \in L^{2}_x(\mathbb{R})$.

\begin{proof}[Proof of Corollary \ref{c:pointwise}]
Fix $\beta<2$ and $\gamma_0\in\mathcal{C}^\beta(\dot{H}^\frac14(\mathbb{R}))$, and let $\gamma(t) = e^{-it\partial_x^2} \gamma_0 e^{it\partial_x^2}$. Clearly, it suffices to prove
\begin{equation} \label{e:pointwisegoal}
\| \limsup_{t\to 0} |\rho_{\gamma(t)} - \rho_{\gamma_0}| \|_{L^{2,\infty}_x} =0.
\end{equation}
As in the discussion preceding this proof, we approximate $\gamma_0 =  \sum_{j=1}^\infty \nu_j \Pi_{g_j}$ by the finite-rank operator $\gamma_0^N = \sum_{j=1}^N \nu_j \Pi_{g_j}$, and define $\gamma^N(t) = e^{-it\partial_x^2} \gamma^N_0 e^{it\partial_x^2}$. Then we claim that  
\begin{equation}\label{e:0711-1}
	\lim_{N \to \infty} \| \rho_{\gamma(t)} - \rho_{\gamma^N(t)} \|_{L^{2,\infty}_xL^\infty_t} = 0.
\end{equation}
To see this we make use of Theorem \ref{t:maximal} as follows. 
By \eqref{e:max}, we have
\begin{align*}
	\| \rho_{\gamma(t)} - \rho_{\gamma^N(t)}\|_{L^{2,\infty}_xL^\infty_t} 
	&= \bigg\| \sum_{j=N+1}^\infty \nu_j |e^{it\partial_x^2}g_j|^2\bigg\|_{L^{2,\infty}_xL^\infty_t}\\
	&\lesssim \bigg( \sum_{j=N+1}^\infty |\nu_j|^\beta\bigg)^\frac1{\beta}. 
\end{align*}
Since $\gamma_0\in\mathcal{C}^\beta(\dot{H}^\frac14(\mathbb{R}))$ we have $\nu \in\ell^{\beta}$ and hence \eqref{e:0711-1} follows.
  
From the definition of $\rho_{\gamma_0}$ and \eqref{e:0711-1}, for any $\varepsilon>0$, we can find $N_\varepsilon$ such that 
\[
\| \rho_{\gamma_0} - \rho_{\gamma_0^{N_\varepsilon}}\|_{L^{2}_x},\ \| \rho_{\gamma(t)} - \rho_{\gamma^{N_\varepsilon}(t)} \|_{L^{2,\infty}_xL^\infty_t} <\varepsilon. 
\]
For such $N_\varepsilon$, we have 
\begin{align*}
&\,\,\quad \| \limsup_{t\to 0} |\rho_{\gamma(t)} - \rho_{\gamma_0}| \|_{L^{2,\infty}_x} \\
& \leq
\| \limsup_{t\to 0} |\rho_{\gamma(t)} -\rho_{\gamma^{N_\varepsilon}(t)}| \|_{L^{2,\infty}_x} + \| \limsup_{t \to 0} |\rho_{\gamma^{N_\varepsilon}(t)} - \rho_{\gamma^{N_\varepsilon}_0}| \|_{L^2_x} + \| \rho_{\gamma^{N_\varepsilon}_0} - \rho_{\gamma_0} \|_{L^2_x} \\ 
& \leq
2\varepsilon + \| \limsup_{t \to 0} |\rho_{\gamma^{N_\varepsilon}(t)} - \rho_{\gamma^{N_\varepsilon}_0}| \|_{L^2_x}.
\end{align*}
Since $g_j\in \dot{H}^{\frac14}(\mathbb{R})$, it follows from Carleson's result in \cite{carleson} that
\[
\limsup_{t \to 0} \rho_{\gamma^{N_\varepsilon}(t)}(x) = \sum_{j=1}^{N_\varepsilon} \nu_j \limsup_{t\to0} |e^{it\partial_x^2} g_j(x)|^2 = \sum_{j=1}^{N_\varepsilon} \nu_j |g_j(x)|^2 = \rho_{\gamma^{N_\varepsilon}_0}(x)
\]
holds almost everywhere. Hence $\| \limsup_{t \to 0} |\rho_{\gamma^{N_\varepsilon}(t)} - \rho_{\gamma^{N_\varepsilon}_0}| \|_{L^2_x}=0$ and we obtain \eqref{e:pointwisegoal}.
\end{proof}

\section{Additional remarks}  \label{section:remarks}

\subsection{Carleson's problem with data in Besov spaces}
Even though the Sobolev regularity $1/4$ in the classical version of Carleson's problem \eqref{e:carleson} is the optimal one, 
it still seems plausible to obtain a further refinement of the estimate \eqref{e:maxi}, in particular, with data in the Besov spaces  $\dot{B}^{1/4}_{2,2\beta}$. For $\beta>1$, we have $\dot{H}^{1/4} \subset \dot{B}^{1/4}_{2,2\beta}$ and thus we would see an improvement in the classical results on Carleson's pointwise convergence problem in the one-dimensional case. Although we are not able to answer this question here, we may quickly obtain the following related result as an additional  application of Theorem \ref{t:maximal}. 
\begin{proposition} \label{p:Besov}
Suppose $\beta<2$. Then the estimate 
\begin{equation}\label{e:weakBesov}
\| e^{it\partial_x^2} f \|_{L^{4,\infty}_x {\rm BMO}_t(\mathbb{R}^{1+1})} \le C \| f \|_{\dot{B}^{\frac14}_{2,2\beta}}
\end{equation}
holds for all $f \in \dot{B}^{\frac14}_{2,2\beta}$.
\end{proposition}
Unfortunately, $\mathrm{BMO}$ is strictly larger than $L^\infty$. (For the definition of $\mathrm{BMO}$ and the homogeneous Besov spaces $\dot{B}^s_{p,q}$, we refer the reader to \cite{SteinHA}.)
\begin{proof}[Proof of Proposition \ref{p:Besov}]
For each $j\in\mathbb{Z}$, let $P_j$ denote the frequency projection operator (with respect to the spatial variable) given by
\[
\widehat{P_jf}(\xi) = \varphi(2^{-j} \xi)\widehat{f}(\xi),
\]
where $\varphi \in C^\infty_c([-4,4]\setminus [-\frac14,\frac14])$ is chosen such that $\varphi\equiv 1$ on $[-2,2]\setminus[-\frac12,\frac12]$ and $\sum_{j \in \mathbb{Z}} \varphi(2^{-j}\xi) = 1$ for $\xi \neq 0$. Also, we let $Q_j$ denote the frequency projection operator (with respect to the temporal variable) given by
\[
\widehat{Q_jf}(\tau) =  \theta(2^{-j} \tau)\widehat{f}(\tau),
\]
where $\theta$ is a similarly chosen bump function which satisfies $\theta \equiv 1$ on $[-16,16]\setminus[-\frac{1}{16},\frac{1}{16}]$.

For each fixed $x \in \mathbb{R}$, the support of the (temporal) Fourier transform of $t \mapsto e^{i t \partial_x^2} P_j f(x)$ is contained in $\{ \tau \in \mathbb{R} : |\tau| \in [2^{2j - 4}, 2^{2j + 4}]\}$ and therefore the function $t \mapsto e^{it\partial_x^2}P_jf(x)$ is invariant under the action of $Q_{2j}$. It follows from this and the Littlewood--Paley inequality in the temporal variable that 
	\begin{align*}
	\| e^{it\partial_x^2} f \|_{L^{4,\infty}_x {\rm BMO}_t}  &= \bigg\| \sum_{j\in\mathbb{Z}} Q_{2j}(e^{i t \partial_x^2}P_jf) \bigg\|_{L^{4,\infty}_x {\rm BMO}_t} \\
	&\lesssim \bigg\| \bigg( \sum_{j\in\mathbb{Z}} |Q_{2j}(e^{i t\partial_x^2}P_jf)|^2 \bigg)^\frac12 \bigg\|_{L^{4,\infty}_x L^\infty_t} \\
	& \lesssim \sum_{k=0}^3 \bigg\| \bigg(\sum_{j\in 4\mathbb{Z} + k} |e^{it\partial_x^2}P_jf|^2 \bigg)^\frac12\bigg\|_{L^{4,\infty}_x L^\infty_t}.
	\end{align*}
For each fixed $k = 0,1,2,3$, it is readily checked that $(P_jf/\|P_jf\|_{\dot{H}^\frac14})_{j\in 4\mathbb{Z}+k}$ forms an orthonormal system, and thus Theorem \ref{t:maximal} implies
$$
	\bigg\| \bigg(\sum_{j\in 4\mathbb{Z} + k} |e^{it\partial_x^2}P_jf|^2 \bigg)^\frac12\bigg\|_{L^{4,\infty}_x L^\infty_t} \lesssim \bigg( \sum_{j\in 4\mathbb{Z} + k}\|P_jf\|_{\dot{H}^\frac14}^{2\beta}  \bigg)^\frac1{2\beta},
$$
from which we obtain \eqref{e:weakBesov}. 
\end{proof}

\subsection{Other dispersion relations}
For simplicity of the exposition, we have stated our main results in the Introduction in terms of the classical Schr\"odinger operator $U_2$. However, an inspection of our proofs of Theorems \ref{t:maximal} and \ref{t:main} reveal that generalization to a wider class of dispersive equations is possible with straightforward modifications. As a concrete example, Theorem \ref{t:fractional-end} may be generalized to the statement that, for $a > 1$ and $\beta < 2$, we have
\[
\bigg\|  \sum_{j}\nu_j |U_{\frac1a}|\partial_x|^{-\frac{2a-1}{4a}} f_j |^2 \bigg\|_{L^{2,\infty}_tL^\infty_x(\mathbb{R}^{1+1})} \lesssim \| \nu \|_{\ell^\beta}
\]
for all systems of orthonormal functions $(f_j)_j$ in $L^2(\mathbb{R})$ and $\nu=(\nu_j)_j$ in $\ell^\beta$. It is also clear that the identity \eqref{e:change} may be appropriately modified to relate $U_a$ with $U_{1/a}$ with the roles of space and time reversed, and consequently we may deduce that the estimates in Theorems \ref{t:maximal} and \ref{t:main} hold with $U_2$ replaced by $U_a$ for $a > 1$.

\subsection{Further discussion} Finally we make additional comments regarding possible development in different directions. 
\begin{enumerate}
[leftmargin=0cm, labelsep=4pt, parsep=7pt, itemindent=15pt]
\item[$(\rm I)$] (Higher dimensions)
Concerning the classical form of Carleson's problem \eqref{e:carleson} in higher dimensions, Bourgain \cite{Bourgain-point2} showed the necessary regularity condition $s\ge \frac12 - \frac1{2(d+1)}$, and Du--Guth--Li \cite{Du-Guth-Li} ($d=2$) and Du--Zhang \cite{Du-Zhang} ($d\ge3$) recently proved that the condition $s> \frac12 - \frac1{2(d+1)}$ suffices, thus leaving open only the endpoint case. Their essentially definitive results built on a number of significant prior work including, for example, \cite{Bourgain-point1,Lee-point,Luca-Rogers-1,Luca-Rogers-2}. 
There are also numerous results on variants of Carleson's problem; see \cite{Bailey,ChoKo,CLV,DS,LY,Shiraki,Sjolin1,Sjolin2,SS}. We believe it is an interesting problem to extend Corollary \ref{c:pointwise} to higher dimensions. However, the arguments in \cite{Du-Guth-Li,Du-Zhang} are very far from the ones we have used in this article, and obtaining a sharp version of \eqref{e:carleson-density} for higher dimensions looks very challenging.

\item[$(\rm I\!I)$](Nonlinear equations)
We considered the pointwise convergence \eqref{e:carleson-density} for the free solution $\gamma(t) = e^{-it\partial_x^2} \gamma_0 e^{it\partial_x^2}$ as a first step toward generalization of the classical maximal estimates to orthonormal systems of initial data.  However, from the perspective of the quantum mechanics, it is more natural to consider \eqref{e:carleson-density} with $\gamma(t)$ which is a solution of the nonlinear equation \eqref{e:Hartree-general}. Related to this problem,  we note that Compaan--Luc\'{a}--Staffilani \cite{CLS} recently investigated the behavior of the solution to the nonlinear Schr\"{o}dinger equation as $t\to0$. 

\item[$(\rm I\!I\!I)$](Improving summability with higher regularity) 
Our pointwise convergence result in Corollary \ref{c:pointwise} is given under the optimal regularity assumption $s=\frac14$. However, there is no reason to restrict ourselves to the specific regularity exponent when dealing with orthonormal systems of initial data. In fact, it seems to be natural to expect a gain of summability in the exponent $\beta$ by imposing higher regularity. Such kind of tradeoff between regularity and summability has been already observed in \cite{BHLNS}. The problem of characterizing $\beta = \beta(s)$ for which the pointwise convergence \eqref{e:carleson-density} holds for $\gamma_0 \in \mathcal{C}^\beta(H^s(\mathbb{R}))$ remains open.  Corollary \ref{c:pointwise} only ensures that $\beta < 2$ is sufficient for all $s\ge\frac14$. 

\end{enumerate}

\begin{acknowledgements}
This work was supported by JSPS Kakenhi grant numbers 18KK0073 and 19H01796  (Bez), Korean Research Foundation Grant  no.  NRF-2018R1A2B2006298 (Lee), and Grant-in-Aid for JSPS Research Fellow no. 17J01766 (Nakamura). 
\end{acknowledgements}

\end{document}